\documentclass{article}
\usepackage{amsmath,amssymb,amsthm}
\usepackage{bbm}
\usepackage{graphicx}

\usepackage{authblk}

\usepackage{algorithm}
\usepackage{algpseudocode}
\usepackage{algorithmicx}
\usepackage{multicol}
\usepackage{bm}

\newtheorem{thm}{Theorem}[section]

\newtheorem{prop}[thm]{Proposition}
\theoremstyle{definition}
\newtheorem{defn}[thm]{Definition}
\theoremstyle{remark}

\numberwithin{equation}{section}
\numberwithin{algorithm}{section}

\DeclareMathOperator*{\argmin}{arg\,min}

\author{Myun-Seok Cheon}
\affil{Corporate Strategic Research, \\ ExxonMobil Research and Engineering}
\title{An Outer-approximation Guided Optimization Approach for Constrained Neural Network Inverse Problems}

\begin{document}
\maketitle
\newcommand{\N}{\mathcal{N}} 
\newcommand{\R}{\mathbb{R}} 
\newcommand{\K}{\mathcal{K}} 
\newcommand{\tr}{\mathcal{\gamma}} 
\newcommand{\x}{\bm{x}}

\begin{abstract}
	This paper discusses an outer-approximation guided optimization method for constrained neural network inverse problems with rectified linear units. The constrained neural network inverse problems refer to an optimization problem to find the best set of input values of a given trained neural network in order to produce a predefined desired output in presence of constraints on input values. This paper analyzes the characteristics of optimal solutions of neural network inverse problems with rectified activation units and proposes an outer-approximation algorithm by exploiting their characteristics. The proposed outer-approximation guided optimization comprises primal and dual phases. The primal phase incorporates neighbor curvatures with neighbor outer-approximations to expedite the process. The dual phase identifies and utilizes the structure of local convex regions to improve the convergence to a local optimal solution. At last, computation experiments demonstrate the superiority of the proposed algorithm compared to a projected gradient method.   
\end{abstract}

\section{Introduction}
Neural networks are the most essential and widely used ingredient of modern machine learning applications \cite{SCHMIDHUBER15}. The primary reason of the prevalent usages is the flexibility and capability of approximating any continuous function, known as the \emph{universal approximation theorem}  \cite{HORNIK90}. Deep neural networks have been demonstrating the strength of universal approximation capabilities in many machine learning applications including image processing, natural language processing and reinforcement learning problems. In the image classification problems, convolutional neural networks are commonly adopted to describe the spatial and temporal dependencies in images without any tailored feature engineering \cite{Krizhevsky12}. Recurrent neural networks are another example of the wide usage of analyzing temporal data \cite{Hochreiter97}. In reinforcement learning, various types of neural networks are employed to approximate intricate value and policy functions \cite{Silver17,Peters08}.

Neural network inverse problems refer to a class of optimization problems that find a right set of input parameters to achieve a desired output with a trained neural network. The trained neural networks can be viewed as a surrogate function that predicts outputs of an input. For example, in the material design context, neural networks can be trained to predict material properties or performances with material fingerprints such as molecular structures. The inverse problem is to find the right molecular structure to produce a certain material property. The input parameters are the main optimization variables for the inverse problems whereas the forward problems optimize weights and biases of a neural network with a large amount of input and output data to improve the prediction accuracy. 

\subsection{Constrained neural network inverse problems}
Let $f(\x): \mathbb{R}^n \rightarrow \mathbb{R}^m$ be a neural network function whose input and output dimensions are $n$ and $m$, respectively. Let $\bm{\hat{f}} \in \mathbb{R}^m$ be a desired output. Let $\mathcal{L}(\cdot,\cdot): \mathbb{R}^m \times \mathbb{R}^m \rightarrow \mathbb{R}$ be a function to measure the difference or loss between two vectors. Let $\mathcal{X}$ be the feasible space for input values. The constrained neural network inverse problem can be described as follows;
\begin{eqnarray}
(\text{P}) \qquad \min_{ x \in \mathcal{X} } & &  \mathcal{L}(f(x), \bm{\hat{f}}) \label{eq:nninv}
\end{eqnarray} 
The optimization problem is to find the input values ($\x$) that minimize the difference between the desired output ($\bm{\hat{f}}$) and the corresponding neural network output ($f(\x)$). For the sake of simplicity, let $g(\x) : \mathbb{R}^n \rightarrow \mathbb{R}$ denote $\mathcal{L}(f(\x), \hat{f}) \in \mathbb{R}$ and $\nabla g(\x) \in \mathbb{R}^n$ be the gradient of  $\mathcal{L}(f(\x), \bm{\hat{f}})$ at $\x$. 

\subsection{Potential applications}
Adversarial examples for a trained neural network are crafted to increase robustness of neural networks. The main idea is from the observation that trained neural networks often make mistakes on classification or prediction tasks for images that have a small perturbation of the original image, which are indifferent to human eyes \cite{Szegedy14,Linden89}. Such adversarial examples in the training set improve the robustness. Generating adversarial examples can be formulated as a neural network inverse problem. For a given image and the corresponding classification, the main decision variable of the problem is perturbations of the image and the main objective is to minimize the prediction accuracy to the corresponding classification. In this problem, the magnitude of perturbation can be controlled by constraints \cite{Fischetti18,Anderson19}.

The second example is any engineering and scientific analysis that consists of forward and inverse modeling. For example, material design and discovery activities can be described with two major parts; one is the forward (prediction) problem that predicts the property of a given material and the other is the inverse problem that finds a set of materials for a given desired property. With a large amount of data or simulation, deep neural networks can be trained to capture the core characteristics of the forward problem. The inverse problem with the trained neural network can be of the form problem (P). The problem optimizes the input, which is a choice of materials, to achieve a desired output such as a target property \cite{Chen20}. Similar structures of two phase approaches for engineering  problems have been discussed in several literatures \cite{Rezaee19,CORTES09}   

Another motivating example is physics-based inversion problems with a lower-dimensional representation. The physics-based inverse problems are to find the right set of parameters for a physics model that explain observations. For example, X-ray tomography is a nondestructive method to analyze the internal property and structure of an object with X-ray measurements. Another example is the subsurface analysis in the oil and gas industries, which utilize seismic and gravity data to describe the subsurface characterization. One of main challenges on such analysis is the ill-posedness of the inverse problem such that the number of model parameters to fit are several orders of magnitude larger than the number of independent observations. This phenomenon causes overfitting such that the resulting model explains the observation data extremely well but it is not a physically plausible solution. One idea to circumvent the issue is incorporating additional information through machine learning techniques into the inversion process. The main role of the machine learning techniques such as variants of variational autoencoders is to capture the main characteristics of reasonable physical realizations with many prior examples. Variational autoencoders consist of an encoder that maps or abstracts physical realizations to latent space variables and a decoder that projects latent variables to a physical realization. The decoder of a successfully trained variational autoencoder can produce a plausible physical realization with any value of the latent space along with a probabilistic measure. In the new approach, the resulting problem is of the form a neural network inverse problem that combines the physics based inverse problem along with the trained decoder. The problem optimizes the latent space variables of the decoder to minimize the difference between the given observation and the physics response of the corresponding physics model from the decoder \cite{omalley19}.  
     
\subsection{Prior solution methods}
Solution techniques for neural network inverse problems can be categorized into three classes; one group is methods that utilize the gradient information such as steepest descent algorithms, projected gradient method, etc. The second group is methods based on the forward function calls such as derivative-free methods and meta-heuristics. The last group is based on explicit mathematical formulations such as mixed integer nonlinear programming.

The gradient respect to input parameters can be computed with back-propagation of neural networks \cite{Linden89}. Unconstrained problems, i.e., no additional constraints for the input parameters, can be tackled with the gradient based methods such as steepest descent methods, quasi-Newton methods, etc. For constrained problems, the projected gradient method can be considered.
 
Meta-heuristics such as particle swarm optimization or genetic algorithms can be considered. The meta-heuristics are attractive because of inexpensive forward evaluation of trained neural networks. Compared to gradient based approaches, these methods avoid trapping into a local optimum \cite{Rezaee19}.
 
When a neural network has only rectified linear activation units, the inverse problem can be described as a mixed integer nonlinear programming problem \cite{Fischetti18,Anderson19}. Section \ref{sec:mp} will discuss a mathematical model.     

\subsection{Contributions}
This paper proposes an optimization algorithm in the context of a multi-start optimization framework for neural network inverse problems. The multi-start concept is employed to mitigate the local optimal issue of neural network inverse problems. In the framework, the initial optimization starting points are collected from the training set, e.g., $n$-closest data for a given target or $n$-clustering center points. There are two benefits of using the training set to select starting points. One is that it is easy to collect meaningful starting points with polynomial time examination. The other is that the solutions stay within the training region. Since neural networks are a surrogate function with a limited training data, if a solution is far from the training data region, the corresponding prediction might not be accurate. By using the trained data for starting solutions, the resulting solution might be retained within the comfort zone of neural networks. The multi-start optimization framework requires solving many optimization problems with various starting points. Therefore, it is important to have an efficient solution method.  

The contributions of this paper are twofold; One is that the paper analyzes the characteristics of local optimal solutions of neural network inverse problems with rectified activation units; The second contribution is the development of an efficient algorithm exploiting the characteristics and a demonstration of the superiority compared to a gradient based method through computational experiments.

The characteristics of neural network inverse problems and the types of local optimal solutions are discussed in Section \ref{sec:nnip}. In Section \ref{sec:ogo}, an outer approximation based algorithm is proposed. Section \ref{sec:comp} contains the computational results for the proposed method and a projected gradient method.  
 
\section{Characteristics of neural network inverse problems} \label{sec:nnip}
\subsection{Mathematical models} \label{sec:mp}
A neural network consists of multiple layers of interconnected neurons. Each neuron is a computing unit defined by weights, bias and an activation function. The weights and bias describe a linear relationship with outputs from connected neurons and the activation function, typically a nonlinear operator and applied after the linear computation, generates the final output. Equation \eqref{eq:relu} describes the computation in a neuron with a rectified linear activation function. 
\begin{equation}
t_j = \max \left (0, \sum_i w_{ij}t_i + b_j \right ), \label{eq:relu}
\end{equation}
where $t_i$ is the output of neuron $i$ and $w_{i,j}$ and $b_j$ are the weights from neuron $i$ to neuron $j$ and the bias term for neuron $j$, respectively. In Equation \eqref{eq:relu}, the rectified linear activation is described with the max operator.
Equation \eqref{eq:relu} can be described as a mixed integer linear system as follows;
\begin{align}
	&t_j - s_j = \sum_{i} w_{ij} t_i + b_j, && \forall j \in N^r, \label{eq:lps} \\
	&t_j \leq \mathbb{M} z_j,  && \forall j \in N^r, \\
	&s_j \leq \mathbb{M} (1 - z_j), && \forall j \in N^r, \\
	&t_j, s_j \geq 0, && \forall j \in N^r, \\	
	&z_j \in \{0,1\}, && \forall j \in N^r, \label{eq:lpe}
\end{align} 
where parameter $\mathbb{M}$ denotes a big-M, which is a large number and set $N^r$ denotes the neurons with a rectified linear activation function. The binary variable $z_j$ indicates whether neuron $j$ is active or not. When neuron $j$ is active such as $z_j = 1$, only variable $t_j$ can take a nonzero value. Otherwise, variable $s_j$ can take a nonzero value. Only a positive output value ($t_j$) is passed to  connected neurons. 

A neural network can be described as a mixed integer linear system as follows. Let $N$ be the all neurons and let $N^I$ and $N^O$ be the neurons in the input and output layers, respectively. 
\begin{align}
	& \eqref{eq:lps} - \eqref{eq:lpe} \\ 
	&t_j = x_j, && \forall j \in N^I, \label{eq:inp} \\
	&y_j = \sum_{i} w_{ij} t_i + b_j, && \forall j \in N^O, \label{eq:out}\\
	&y_j \in \R, && \forall j \in N^O
\end{align}
In the the mixed integer linear system, $x_j$'s are the parameter for the input layer in Constraint \eqref{eq:inp} and the result is the solution of variable $y_j$'s of the linear system for the output layer in Constraint \eqref{eq:out}. 

The formulation can be tightened with an extended formulation or projected cuts from an extended formulation \cite{Anderson19}.  

\subsection{Characteristics of local optimal solutions}
In this section, we discuss the properties of the neural network inverse problem with rectified linear activation units. Figure \ref{fig:lopt} shows a simple example of a neural network and the squared error loss with a given target. The neural network has 5 neurons and its output is a piece-wise linear, which is depicted in the right graph ((a)) of Figure \ref{fig:lopt}. Given a target output, the orange dotted line in (a) of Figure \ref{fig:lopt}, the left graph of Figure \ref{fig:lopt} shows the squared error loss respect to input $x$.  

\begin{figure}[h]
\begin{center}
\begin{tabular}{cc}
\includegraphics[width=0.45\textwidth]{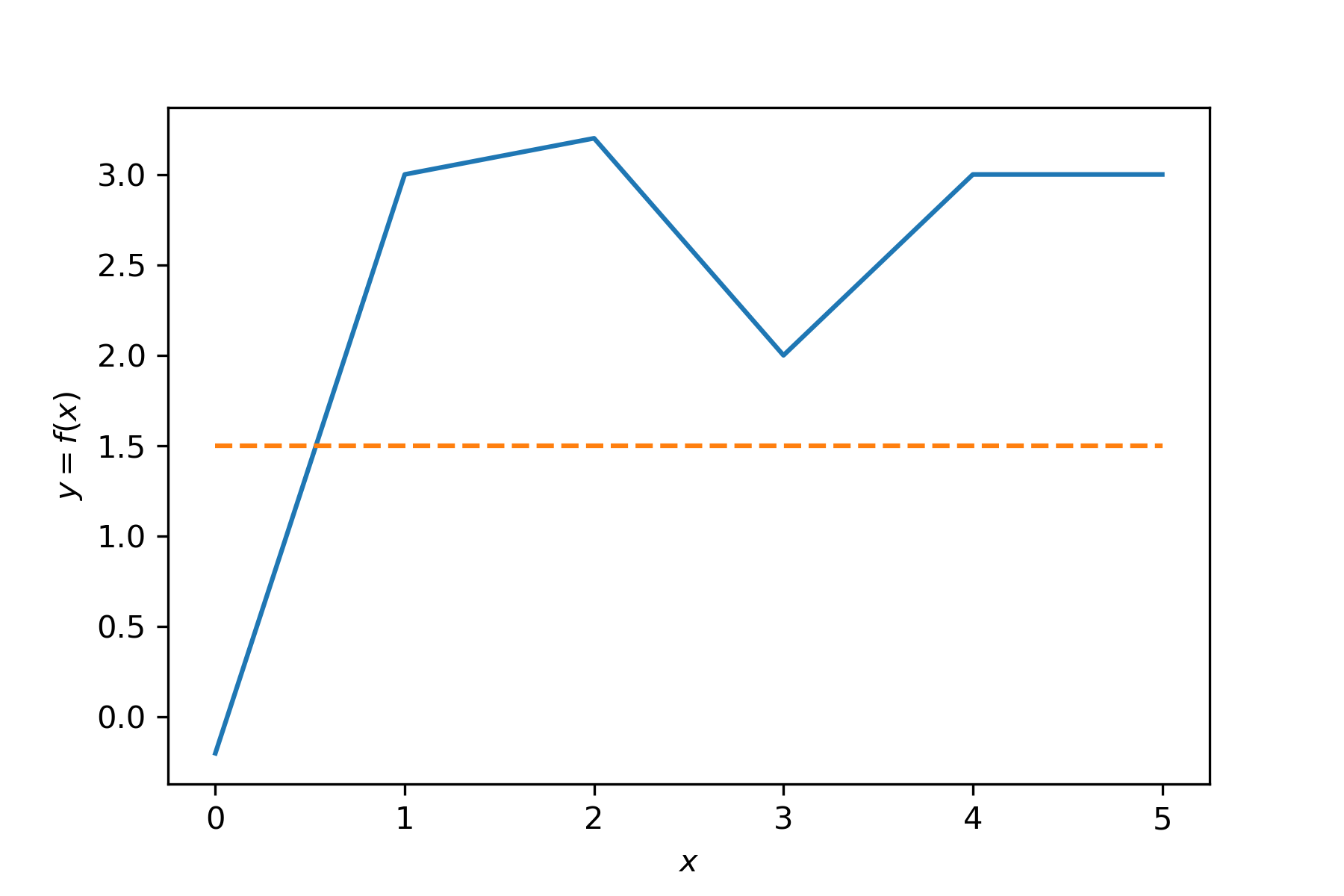}
&
\includegraphics[width=0.45\textwidth]{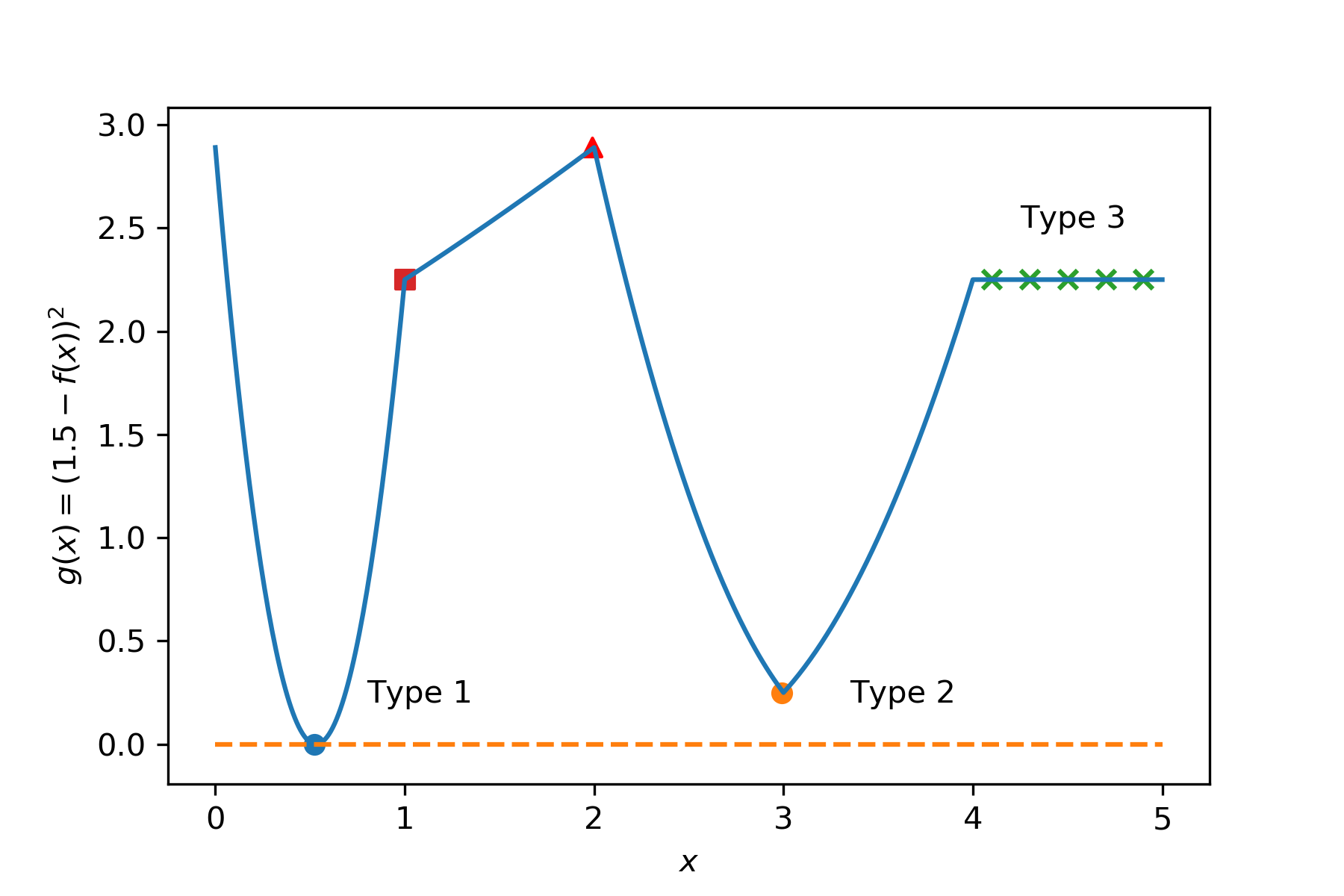} \\
(a) $f(\x)$ & (b) $g(\x)$
\end{tabular}
\end{center}
\caption{An illustrative example of several types of local optimal solutions. (a) a neural network describes a piecewise linear. The blue line represents the output of the neural network over input $x$. Each integer interval has a unique activation pattern. The orange dotted line represents the target for the loss function; (b) the squared error loss function for $f(x)$ with a given target. The figure illustrates various local optimal solutions. The orange dotted line represents the zero loss.} \label{fig:lopt}
\end{figure}

The first observation is that with a given activation such that a set of active neurons is predetermined, the resulting problem is convex. Let $\N(\x)$ be a set of active neurons, whose output is positive for input $\x$. Note that for a given solution $\x$, there can be more than one active set definition such that $\N_1(\x) \neq \N_2(\x)$.
\begin{prop}
If a neural network has only rectified linear activation units and a convex loss function such that $\mathcal{L}(\cdot, \hat{\bm{y}})$ is convex for a given target $\hat{\bm{y}}$, then, the neural network inverse problem  $g(\x)$ over $\x \in \{\x' | \mathcal{N}(\x') = \mathcal{N}(\hat{\x}) \}$ for a given set of activation $\N(\hat{\x})$ is convex.
\end{prop}
\begin{proof}
With a set of fixed activation ($\N(\hat{\x})$), the feasible region of  $\{\x' | \mathcal{N}(\x') = \mathcal{N}(\hat{\x}) \}$ can be described with the following linear system by setting the binary variables of active neurons to one and the rest to zero.
\begin{eqnarray}
\mathcal{Y}(\N(\hat{\x})) &=& \left \{ \x \in \R^n \left | 
\begin{array}{ll}
	t_j = x_j, & \forall j \in N^I \\
	t_j = \sum_{i} w_{ij} t_i + b_j, & \forall j \in \N(\hat{\x}), \\
	0 \geq \sum_{i} w_{ij} t_i + b_j & \forall j \in N^r \setminus \N(\hat{\x}), \\ 
	t_j \geq 0,  & \forall j \in \N(\hat{\x}), \\
	t_j = 0,  & \forall j \in N^r \setminus \N(\hat{\x})
\end{array} \right . \right \} \label{eq:nnls} 
\end{eqnarray}
Since the loss function is convex and the feasible region is convex, the resulting optimization problem is convex.
\end{proof}

Any feasible solution can be categorized into two groups based on the following set definition. Let $\mathcal{B}({\x})$ be a set of neurons whose output after the linear computation, i.e., before the activation operation, is exactly zero for input $\x$ such as 
\begin{equation}
\mathcal{B}({\x}) = \left \{ j \in N^r \left | \sum_{i} w_{ij} \tilde{t}_i + b_j = 0 \right. \right \}, 
\end{equation}
where $\tilde{t}_i$ is the corresponding output of neuron $i$ for input $\x$.
One group is the solutions without any zero output neurons before the activation such that $\mathcal{B}(\x^*) = \emptyset$. The other group is at least one neuron with exactly zero output before the activation such that $\mathcal{B}(\x^*) \neq \emptyset$. Note that any solution is categorized into one of two categories such that $\{ x | \mathcal{B}(x) = \emptyset \land \mathcal{B}(x) \neq \emptyset \} = \emptyset$ and  $\{ x \in \mathcal{X} | \mathcal{B}(x) = \emptyset \lor \mathcal{B}(x) \neq \emptyset \} = \mathcal{X}$. In Figure \ref{fig:lopt}, any integer solution $x \in \{ 1, 2, 3, 4 \}$ has at least one zero output neuron. Set $\mathcal{B}(x)$ for all other real numbers $x \in \R \setminus \{ 1, 2, 3, 4 \}$ is empty.  

The neural network inverse problem with rectified linear units can be viewed as a union of many convex optimization problems with all possible activation permutations. 
\begin{defn} 
(Type 1) An optimal solution $\x^*$ is called a \textit{self-contained} local optimal solution if the solution $\x^*$ is optimal to the neural network inverse problem $g(\x)$ over $\x \in \mathcal{Y}(\N(\x^*))$ and no neuron has an output of exact zero before the corresponding activation such as $\mathcal{B}(\x^*) = \emptyset$.
\end{defn}
`Type 1' solution in Figure \ref{fig:lopt} is an example of a \textit{self-contained} local optimal solution. In this case, it is sufficient to prove the local optimality without considering its neighbor activation regions. If there exist zero output neurons such as $\sum_{i} w_{ij} \tilde{t}_i + b_j = 0$, there exists more than one activation pattern that contains solution $\x^*$. In this case, it is required to check if the solution is optimal respect to all possible permutations of  $\mathcal{N}(\x^*)$. Let $\mathcal{M}(\x^*)$ be a set of all feasible activation permutations for solution $\x^*$. Note that some permutations of activations can be infeasible. 
\begin{defn} 
(Type 2) An optimal solution $\x^*$ is called a \textit{boundary} local optimal solution if the solution $\x^*$ is optimal to the neural network inverse problem $g(\x)$ over $\x \in \{\x' | \mathcal{N}(\x') \in \mathcal{M}(\x^*) \}$ and the solution is associated with more than one activation pattern.
\end{defn}
In this case, there exists at least one neuron with exactly zero output such as $\mathcal{B}(\x^*) \neq \emptyset$. `Type 2' solution in Figure \ref{fig:lopt} falls into this category. When $x=3$, there are two possible activations - one is for the interval from 2 to 3 and the other is for the interval from 3 to 4. In both cases, the solution ($x=3$) is locally optimal. Conversely, consider the solution of $x=1$, which is a red square in Figure \ref{fig:lopt}. The solution $x=1$ is a local optimal respect to the convex region defined by the activation for $1 \leq x \leq 2$ while it is not local optimal for the convex region for the activation $0 \leq x \leq 1$.  
\begin{defn}
(Type 3) An optimal solution $\x^*$ is called a \textit{redundant} local optimal solution if an element $x_j$ of the solution $\x^*$ is insensitive to the local optimality such that solutions with any value $x_j$ within the feasible region $\x \in \{\x' | \mathcal{N}(\x') = \mathcal{N}(\x^*) \}$ are still locally optimal. 
\end{defn}
Note that the \textit{redundant} local optimal solutions can be categorized as a \textit{self-contained} or \textit{boundary} local optimal solution. `Type 3' solution in Figure \ref{fig:lopt} is an example of \textit{redundant} local optimal solutions. The solution is similar to a saddle point in nonlinear optimization problems. Note that the solution ($x=4$) at the boundary of `Type 3' solution in Figure \ref{fig:lopt} is not a local optimal solution. Its left neighbor has better objective values.   

\begin{thm}
	A solution ($\x^*$) is locally optimal to Problem (P) if and only if the solution ($\x^*$) is locally optimal for all neighbor activation subproblems such that
	\begin{equation}
		g(\x^*) = \min_{ \x \in \mathcal{X} \cap \mathcal{Y}(\N(x^*)) } g(\x) \qquad \forall \N(\x^*) \in \mathcal{M}(\x^*).
	\end{equation}  
\end{thm}
\begin{proof}
	Suppose that a solution ($\x^*$) is locally optimal respect to all subproblems defined by its neighbor activation patterns and it is not locally optimal to Problem (P). Then, for any $\epsilon > 0$, there exists a solution $\hat{\x}$ such that $\| \x^* - \hat{\x} \| < \epsilon$, $g(\x^*) > g(\hat{\x})$ and $\x^* \notin \mathcal{Y}(\N(\hat{\x}))$. Since $\mathcal{Y}(\N(\hat{\x}))$ is a closed set (a linear system), there exists $0 < \lambda < 1$  such that $(1 - \lambda) \times \x^* - \lambda \times \hat{\x} \notin \mathcal{Y}(\N(\hat{\x}))$. That is, $\hat{\x}$ is not a neighbor. It contradicts.
	
	Suppose that there is an activation pattern $\N(x^*)$ whose corresponding neighbor convex region contains a better solution ($\bar{\x}$) such as  
	\begin{equation}
		g(\x^*) > g(\bar{\x}) = \min_{ \x \in \mathcal{X} \cap \mathcal{Y}(\N(x^*)) } g(\x).
	\end{equation}
	Since $g(\x)$ is convex, the following inequality holds  for $0 < \lambda \leq 1$; 
	\begin{equation}
		g(( 1 - \lambda ) \times \x^* + \lambda \times \bar{\x}) \leq ( 1 - \lambda ) \times g(\x^*) + \lambda \times g(\bar{\x}) < g(\x^*).
	\end{equation}
	It implies that there exists a neighbor solution $\bar{\x}$ that has a better objective value. Thus, $\x^*$ is not a local optimal solution.   
\end{proof}

\section{Outer approximation guided algorithm} \label{sec:ogo}
In this section, an outer-approximation guided algorithm is proposed for constrained neural network inverse problems. The proposed algorithm adapts core concepts from gradient based algorithms and outer approximation approaches \cite{Geoffrion70} while it exploits the characteristics of local optimal solutions. The proposed algorithm iteratively identifies a descent direction with outer approximation subproblems and determines the next solution with a step size.

The proposed algorithm comprises two phases - primal and dual phases; The primal phase is to improve the solution by incorporating local (neighbor) gradients and the dual phase is focusing on proving local optimality. The algorithm uses two outer approximation subproblems to find a descent direction for primal and dual phases and one outer approximation subproblem to prove local optimality.

\subsection{Outer approximation algorithms}
The outer approximation algorithms are widely used to solve large scale convex optimization and convex mixed integer programming problems\cite{Geoffrion70,Benders62,Duran86}. The main idea of the algorithms is to approximate nonlinear convex or complex linear systems with a set of hyperplanes. Consider the following optimization problem;
\begin{align}
\min_{ x \in \mathcal{X} } \quad & g(x),
\end{align}
where $g(x)$ is convex over $x \in \mathcal{X}$. Then, we can approximate the problem with hyperplanes as follows;
\begin{align}
\min_{x \in \mathcal{X}, v \in \mathbb{R}} \quad & v \\
\text{s.t.} \quad & v \geq g(x^k) + \nabla g(x^k)^T (x - x^k) & & k \in \K,  
\end{align}
where $x^k$ is feasible such that $x^k \in \mathcal{X}$ and $\nabla g(x^k)$ is the gradient at $x^k$. When $g(\cdot)$ is a certain linear system, the dual values can be used instead of gradients \cite{Benders62}. Similar approaches have been proposed and applied for large-scale nonsmooth convex optimization problems \cite{BenTal05}.  

The outer-approximation is valid only when function $g(x)$ is convex over the feasible set $\mathcal{X}$. One can easily show that the neural network functions with rectified linear activation functions are not convex. Even though the approximation is not valid, it can be used to determine a descent direction. The approximation will provide additional information from previous observations compared to simple first order methods. In addition, if the final solution is within a localized convex region, the corresponding approximation on the region is valid.

\subsection{Outer approximation for neural network inverse problems} \label{sec:noa}
In this section, we discuss an outer approximation model for the neural network inverse problems. Let $\K$ be an index set for solutions.  Let set $\mathcal{K}(\x')$ be an index subset for solutions ($\x^k$) in the union of convex feasible regions defined by activation patterns of solution $\x'$ such that
\begin{eqnarray}
\mathcal{K}(\x') &=& \{ k \in \mathcal{K}| \exists \N(\x^k) \in \mathcal{M}(\x') \}. 
\end{eqnarray}
Let $\x^*$ be the best solution among solutions $\x^k$'s such that $\x^* = \argmin_{\x^k,  k \in \K } g(\x^k)$.
Consider an outer approximation model for neural network inverse problems.  
\begin{align}
(\text{NOA}) \qquad \min_{\x \in \mathcal{X}, v \in \mathbb{R}} \quad & v \\
\text{s.t.} \quad & v \geq g(\x^{k}) + \nabla g(\x^{k})^T (\x - \x^{k}), & & {k} \in \K(\x^*). \label{eq:oa} 
\end{align} 
Note that $g(\x^*)$ is not differentiable when $\mathcal{B}(\x^*) \neq \emptyset$. Let $\nabla_{\N(\cdot)} g(\x^*)$ be the gradient of function $g(\cdot)$ at solution $\x^*$ in a specific activation pattern $\N(\cdot)$. With an activation pattern $\N(\cdot)$, the resulting problem such as minimizing $g(\x)$ over $\x \in \mathcal{Y}(\N(\cdot))$ is convex and differentiable for any feasible solution.
\begin{prop} \label{prop:lochk}
	Let $\x^*$ be a solution in the convex region defined by an activation pattern $\N(\x^*)$. The solution $\x^*$ is locally optimal for the convex region $\mathcal{Y}(\N(\x^*))$ if and only if the following inequality holds;
	\begin{equation}
	\nabla_{\N(\cdot)} g(\x^*)^T ( \x - \x^* ) \geq 0, \qquad \forall \x \in \mathcal{Y}(\N(\x^*)).	\label{eq:optchk}
	\end{equation}
\end{prop}
\begin{proof}
Let $\x^*$ be a local optimal solution within convex region $\mathcal{Y}(\N(\x^*))$. It implies that there is no feasible descent direction $\bm{d} = \alpha(\x - \x^*)$ for any feasible $\x \in \mathcal{Y}(\N(\x^*))$ and a scalar $\alpha > 0$. That is, the inequality \eqref{eq:optchk} is valid for all $\x \in \mathcal{Y}(\N(\x^*))$.

Conversely, suppose that the inequality \eqref{eq:optchk} is valid for all $\x \in \mathcal{Y}(\N(\x^*))$ and $\x^*$ is not local optimal. Then, there exists at least a solution $\hat{\x}$ such that $g(\x^*) > g(\hat{\x})$. Since the problem is convex, the following outer-approximation is always valid.
\begin{eqnarray*}
	& & g(\hat{\x}) \geq g(\x^*) + \nabla_{\N(\cdot)} g(\x^*)^T ( \hat{\x} - \x^* ), \\ 
	\Rightarrow & & 0 > g(\hat{\x}) - g(\x^*) \geq \nabla_{\N(\cdot)} g(\x^*)^T ( \hat{\x} - \x^* ).
\end{eqnarray*}
It contradicts that the inequality \eqref{eq:optchk} is valid. Therefore, $\x^*$ is local optimal.
\end{proof}

\begin{thm} \label{thm:opt}
Let $v^*$ be the optimal objective value of the outer approximation (NOA) with solutions $\x^k \in \mathcal{X}, k \in \mathcal{K}$ and $\x^*$ be the best known solution respect to function $g(\x)$ such that $\x^* = \argmin_{\x^k, k\in \mathcal{K}} g(\x^k)$. The solution $\x^*$ is local optimal to Problem (P) if the following conditions hold;
\begin{enumerate}
	\item $g(\x^*) = v^*$,
	\item $\nabla_{\N(\cdot)} g(\x^*)^T ( \x - \x^* ) \geq 0, \forall \x \in \mathcal{Y}(\N(\cdot)),   \N(\cdot) \in \mathcal{M}(\x^*)$.  
\end{enumerate}
\end{thm}
\begin{proof}
Consider two cases; one is $\mathcal{B}(\x^*) = \emptyset$ and the other is $\mathcal{B}(\x^*) \neq \emptyset$.
\begin{enumerate}
\item[i.] $\mathcal{B}(\x^*) = \emptyset$ \\
Consider the feasible set \eqref{eq:nnls}. Since $\mathcal{B}(\x^*) = \emptyset$, the following two constraints of \eqref{eq:nnls} are non-binding; 
\begin{align}
& 0 \geq \sum_{i} w_{ij} t_i + b_j, & & \forall j \in N^r \setminus \N(\hat{x}), \\ 
& 0 \leq t_j = \sum_{i} w_{ij} t_i + b_j,  & & \forall j \in \N(\hat{x}).
\end{align}
The rest of \eqref{eq:nnls} is redundant to function $g(\x)$. Therefore, the outer approximation \eqref{eq:oa} is a binding outer approximation for $g(\x)$ around solution $\x^*$. It means that $\x^*$ is local optimal and \emph{self-contained} local optimal.
\item[ii.] $\mathcal{B}(\x^*) \neq \emptyset$ \\
In this case, it is required to show that the current solution ($\x^*$) is also local optimal to its neighbor activations ($\mathcal{M}(\x^*)$). By Proposition \ref{prop:lochk}, the second condition ensures that the solution is local optimal for feasible regions defined by the corresponding activation patterns. 
\end{enumerate}
\end{proof}
For a given activation pattern $\N(\cdot)$, the following linear programming problem can check if the corresponding feasible region has a descent direction. 
\begin{align}
(\text{OC}(\N(\cdot))) \qquad \min_{\x \in \mathcal{X}} \quad & \nabla_{\N(\cdot)} g(\x^*)^T ( \x - \x^* ) \\
\text{s.t.} \quad & \x \in \mathcal{Y}(\N(\cdot))
\end{align} 
If the optimal objective is equal to zero, then it proves that there is no descent direction.

\subsection{Localized outer approximation subproblems} 
The proposed algorithm consists of primal and dual phases. The primal phase is designed to find a better solution quickly while the focus of the dual phase is to prove the local optimality of the best known solution. The dual phase employs the outer approximation model (NOA) discussed in Section \ref{sec:noa}. In order to identify the proper hyperplanes, the (NOA) outer approximation models require bookkeeping efforts of activation patterns for each solution. In order to alleviate the bookkeeping efforts in the early stage, the primal phase uses an outer approximation constructed based on the simple distance to the best known solution.

\subsubsection{Distance-localized outer approximation subproblem} 
Subproblem (DLOA$^k$) at iteration $k$ is defined as follows;
\begin{align}
(\text{DLOA}^k) \qquad \min_{\x \in \mathcal{X}, v \in \mathbb{R}} \quad & v \\
\text{s.t.} \quad & v \geq g(\x^{k'}) + \nabla g(\x^{k'})^T (\x - \x^{k'}) & & {k'} \in \K^* \cup \{k\},
\end{align}  
where set $\K^*$ is an index subset of hyperplanes that satisfy the following two conditions. One is the base solutions for the selected hyperplanes should be close to the best known solution such that the Euclidean distance from each base solution for hyperplanes to the best known solution should be less than $\gamma^c$.  The second condition is that the resulting hyperplane should not exclude the best feasible solution. The set $\K^*$ is defined as follows;
\begin{eqnarray}
\mathcal{K}^* &=& \left \{{k} \in \K \left | \begin{array}{l} g(\x^*) < g(\x^{k}) + \nabla g(\x^{k})^T (\x^* - \x^{k}), \\ \| \x^* - \x^{k} \| \leq \tr^c \end{array} \right. \right \}.
\end{eqnarray}
Note that since the algorithm drops previously generated outer approximation constraints, it can experience a \textit{stalling} behavior such that the algorithm could revisit the solution that is previously examined. Such behaviors are prevented with two mechanism. The first mechanism is to always include the latest hyperplane even if it cuts off the best known solution. The other mechanism is adjusting the step size. If the algorithm fails to find a better solution, it decreases the step size, which guarantees a different solution. Suppose that the hyperplane defined by the latest solution cuts off the best feasible solution. There are two possible outcomes. One is that the new approximation identifies a better solution. In this case, there is no stalling problem. The other case is that the new approximation fails to identify a new solution. In this case, the newly added hyperplanes will be removed with the validity check against the best known solution. It can lead to the same approximation model as before. However, since the algorithm fails to find a better solution, it will reduce the step size, which prevents the \emph{stalling} behavior.

\subsubsection{Region-localized outer approximation subproblem}

The following subproblems are employed in the dual phase to expedite proving the local optimality. Let $\N(\cdot) \in \mathcal{M}(\x^*)$ be a selected activation pattern. 
\begin{align}
(\text{RLOA}^k(\N(\cdot))) \qquad \min_{\x \in \mathcal{X}, v \in \mathbb{R}} \quad & v \\
\text{s.t.} \quad & v \geq g(\x^{k'}) + \nabla g(\x^{k'})^T (\x - \x^{k'}) & & \forall \x^{k'} \in \mathcal{Y}(\N(\cdot)), \label{eq:chp}\\
&  0 \geq { \bm{r}^{k'}}^T \x + d^{k'} & & \forall \x^{k'} \notin \mathcal{Y}(\N(\cdot)), \label{eq:fhp} 
\end{align}
where Constraint \eqref{eq:fhp} are feasibility cuts under consideration of the convex region $\mathcal{Y}(\N(\cdot))$. Since the set $\mathcal{Y}(\N(\cdot))$ is a linear system, a feasibility cut for solution $\x^{k'} \notin \mathcal{Y}(\N(\cdot))$ can be constructed based on the extreme ray of the following dual problem.
\begin{align}
\max_{r, \pi} \quad & \sum_{j \in N^I} x^{k'}_j r_j + \sum_{j \in  N \setminus N^I} b_j \pi_j \\
& \sum_{j \in N^I} x^{k'}_j r_j + \sum_{j \in  N \setminus N^I} b_j \pi_j  \leq 1, \\
& r_j - \sum_{k \in  N} w_{jk} \pi_k  = 0, & &\forall j \in N^I,  \\
& \pi_j - \sum_{k \in  N} w_{jk} \pi_k \leq 0, & &\forall j \in \N(\cdot),  \\
& r_j \in \mathbb{R}, & &\forall j \in N^I,  \\
& \pi_j \in \mathbb{R}, & &\forall j \in \N(\cdot), \\
& \pi_j \geq 0, & & \forall j \in N^r \setminus \N(\cdot), \\
& \pi_j = 0, & & \forall j \in N^O
\end{align}
When the neural network has many nodes, solving the aforementioned dual problem can be challenging and often numerically unstable. In order to overcome numerical and computational challenges, the proposed algorithm uses forward-propagation to populate Constraint \eqref{eq:fhp}. 

Assume that nodes closer to the input layer are assigned with lower index numbers than ones further away. Let $j^d$ be the lowest indexed node with activation discrepancy between the two solutions such that $j^d = \min ( j \in \N^r | ( j \in \N(\x^{k'}) \land j \notin \N(\cdot) ) \lor (j \notin \N(\x^{k'}) \land j \in \N(\cdot)) )$. Let $L$ be the set of layers. Let $n^l$ be the number of nodes in layer $l \in L$. Let $W^l \in \R^{n^l \times n^{l-1}}$ and $b^l \in \R^{n^l}$ be the weight matrix and the bias column vector for layer $l \in L$, respectively. Let $\bm{I}^{l} \in \R^{n^l \times n^l}$ be a matrix whose diagonal entries corresponding to active nodes in the activation pattern $\N(\cdot)$ are one, otherwise zero;
\begin{equation}
{I}^{l}_{i,j} = \left \{ \begin{array}{ll} 1, & \forall i = j, \nu_i \in \N(\cdot), \\ 0, & \text{otherwise}, \end{array} \right .
\end{equation} 
where ${I}^{l}_{i,j}$ denotes $j^{th}$ element in $i^{th}$ row of $\bm{I}^{l}$ and $\nu_i$ denotes the node index corresponding to $i^{th}$ row.

Suppose $\x$ is not in the same activation region $\N(\cdot)$ such that $\N(\x) \neq \N(\cdot)$. With a known activation pattern, the max operation can be described with matrix $\bm{I}^{l}$. The first layer of forward propagation can be computed as follows;

\begin{equation}
\bm{I}^{1}[ \bm{W}^1 \x + \bm{b}^1 ]
\end{equation}

Let $l^{j_d}$ be the layer for node $j_d$. Then, the forward propagation up to node $j_d$ can be described as follows;
\begin{eqnarray}
\widehat{\bm{W}} &=& \prod_{ l = l^{j_d} }^1 \bm{I}^{l} \bm{W}^l, \\
\hat{\bm{b}} &=& \bm{I}^{l^{j_d}}\bm{b}^{l^{j_d}} + \sum_{ l = l^{j_d} - 1 }^1 \left ( \prod_{ l' = l+1 }^{l^{j_d}} \bm{I}^{l'} \bm{W}^{l'} \right ) \bm{I}^{l} \bm{b}^{l}
\end{eqnarray}
Now, the feasibility constraint can be constructed as follows; suppose $i^d$ be $i^{th}$ row of the corresponding output to discrepancy node $j^d$ such that $\nu_{i^d} = j^d$. Let $\widehat{\bm{W}}_i$ and $\hat{\bm{b}}_i$ denote $i^{th}$ row of matrix $\widehat{\bm{W}}$ and $\hat{\bm{b}}$, respectively.
\begin{eqnarray}
\widehat{\bm{W}}_{i^d} \x + \hat{\bm{b}}_{i^d} &\leq& \bm{0} \quad \text{if } j^d \notin \N(\cdot), \\
\widehat{\bm{W}}_{i^d} \x + \hat{\bm{b}}_{i^d} &\geq& \bm{0} \quad \text{if } j^d \in \N(\cdot), 
\end{eqnarray}
By the construction of $j^d$, solution $\x^{k'}$ violates one of the constraints above.

\subsection{Outer approximation guided algorithm}
Algorithm \ref{alg:ogo} describes the outer approximation guided algorithm. The algorithm starts with an initial point $\x^0$ and hyperparameters in line \ref{alg:hp1} - \ref{alg:hp2}. Parameter $\gamma^s$, $\gamma^s_{\min}$, and $\gamma^s_{\max}$ denote the step size and its lower and upper bounds, respectively. Parameter $\gamma^c$ defines the neighborhood size of distance-localized outer approximation subproblems. Parameter $\rho^c$ and $\rho^e$ are decrease and increase ratios for the step size, respectively. Parameter $\epsilon$ is the local optimality gap for termination and parameter $N$ defines the maximum number of iterations. From line \ref{alg:ini1} to line \ref{alg:ini2}, it initializes the bookkeeping parameters; $\bar{x}^*$ and $\bar{g}^*$ denote the best known solution and the corresponding objective value, respectively. Parameter $k$ tracks the number of iteration and parameter dual-phase takes a boolean value to indicate whether the algorithm is in the dual phase or not.   

In each iteration, the algorithm evaluates the function value and the gradient for the current solution $\x^k$ and updates the corresponding constraint for subproblems (line \ref{alg:eval1} - \ref{alg:eval2}). In line \ref{alg:stup1} to \ref{alg:stup2}, if the algorithm finds a better solution than the current best known solution, it updates the best known solution and increases the step size with parameter $\rho^e$. When there is no improvement, it reduces the step size with parameter $\rho^c$. The step size is truncated by predetermined parameters $\gamma^s_{\min}$ and $\gamma^s_{\max}$. Line \ref{alg:it} updates the iteration count. If the step size is too small (Line \ref{alg:reach_min}), the  algorithm sets to the dual phase. When the algorithm is not in the dual phase, it solves the distance-localized outer approximation subproblem to find the next direction. If the objective of the approximation is close to the objective value of the best known solution, it reevaluates the approximation with solutions within the neighbor activation regions. If the resulting approximation is close to the best known objective within $\epsilon$, the algorithm activates the dual phase. In the dual phase, the algorithm terminates if the best known solution is $\epsilon-$optimal for all possible neighbor feasible region. If the algorithm finds a neighbor region with potential improvement, it generates a next direction with the solution of (RLOA$^k$(${\N}(\cdot)$)). In Line \ref{alg:stup}, the algorithm sets the next solution under consideration of the approximation solution $\hat{\x}^k$ and the step size parameter $\gamma^s$. 

\begin{algorithm}[h]
\begin{multicols}{2}
\begin{algorithmic}[1]
\State $\x^0 \in \mathcal{X}$
\State Define $\gamma^s >0, \gamma^s_{\min} > 0, \gamma^s_{\max} > 0$ \label{alg:hp1}
\State Define $\gamma^c > 0, 0 < \rho^c < 1, \rho^e > 1$ 
\State Define $\epsilon > 0, N > 0$	 \label{alg:hp2}	
\State $\bar{\x}^* \gets \x^0$, $\bar{g}^* \gets g(\x^0)$ \label{alg:ini1} 
\State $k \gets 0$
\State dual-phase $\gets false$ \label{alg:ini2}
\For{ $k \in \{0, \ldots, N\}$ } 
	 \State compute $g(\x^k)$ and $\nabla g(\x^k)$ \label{alg:eval1}
	 \State update constraints	\label{alg:eval2}
	 \If{ $g(\x^k) < \bar{g}^*$ }  \label{alg:stup1} 
	 	\State $\gamma^s \gets \min( \rho^e \times \gamma^s, \gamma^s_{\max} ) $	
	 	\State dual-phase $\gets false$
	 	\State $\bar{\x}^* \gets \x^k$, $\bar{g}^* \gets g(\x^k)$		 	
	 \Else
	 	\State $\gamma^s \gets  \max( \gamma^s_{\min}, \rho^c \times \gamma^s ) $	
	 \EndIf \label{alg:stup2}
	 \State $k \gets k + 1$ \label{alg:it}
	 \If { $\gamma^s = \gamma^s_{\min}$ } \label{alg:reach_min}
	 	\State dual-phase $\gets true$
	 \EndIf	 
	 \If { $\neg$ dual-phase }
	 	\State $\hat{\x}^k$, $v^k$ $\gets$ solve (DLOA$^k$)
	 	\If { $\bar{g}^* - v^k < \epsilon$ }
	 		\State $\hat{\x}^k$, $v^k$ $\gets$ solve (NOA)
	 		\If { $\bar{g}^* - v^k < \epsilon$ }
	 			\State dual-phase $\gets true$
			\EndIf
	 	\EndIf	
	 \EndIf
	 
	 \If { dual-phase }
	 	\State $\hat{\x}^k \gets \emptyset$
	 	\For{ $\N(\cdot) \in \mathcal{M}(\x^*)$ } \label{alg:enum}
	 		\State $\hat{\x}^k$, $v^k$ $\gets$ solve (RLOA$^k$(${\N}(\cdot)$)) \label{alg:rloa}
	 		\If { $\bar{g}^* - v^k \geq \epsilon$ }	 			
	 			\State \textbf{break}
	 		\EndIf
	 	
	 	\EndFor
	 	\If {$\hat{\x}^k = \emptyset$}
	 		\State \textbf{break} \label{alg:exit}
		\EndIf	 		
	 \EndIf	 
	 \State $ \x^k = \bar{\x}^* + \gamma^s \times ( \hat{\x}^k - \bar{\x}^* ) $ \label{alg:stup}
\EndFor
\end{algorithmic}
\end{multicols}
\caption{Outer approximation guided algorithm} \label{alg:ogo}
\end{algorithm}

Set $\mathcal{M}(\x^*)$ in Line \ref{alg:enum} can have exponentially many activation patterns. In the implementation, the algorithm records the previously examined activation patterns to avoid reexamining same regions, repeatably. The algorithm refreshes the record when it finds a new best solution. In the examination of neighbor patterns, the algorithm filters the activation patterns with the following inequality in a preliminary way.
\begin{equation}
\nabla_{\N(\cdot)} g(\x^*)^T ( \x^k - \x^* ) < 0, \x^k \in \mathcal{Y}(\N(\cdot)), \N(\cdot) \in \mathcal{M}(\x^*)
\label{eq:precheck} 
\end{equation}   
For each activation pattern $\N(\cdot) \in \mathcal{M}(\x^*)$, the algorithm checks whether the direction to all previous observations within the corresponding feasible region is a descent direction or not. If it is a descent direction, the algorithm solves RLOA$^k$(${\N}(\cdot)$) in Line \ref{alg:rloa}. Even if the algorithm fails to find an activation pattern that satisfies the condition \eqref{eq:precheck}, it does not mean that the algorithm identifies an $\epsilon-$optimal solution. The algorithm still needs to solve RLOA$^k$(${\N}(\cdot)$) in Line \ref{alg:rloa} for all possible activation patterns. The main purpose of the preliminary check is to quickly find a region with higher potential of improvement.

It is very challenging to find an exact solution at the boundary. Therefore, in the implementation, the boundary solutions are determined with a numerical tolerance $\tau > 0$ as follows;
\begin{equation}
\hat{\mathcal{B}}({\x}) = \left \{ j \in N^r \left | \left | \sum_{i} w_{ij} \tilde{t}_i + b_j \right | \leq \tau \right. \right \}. 
\end{equation}

\subsection{Convergence}  
\begin{thm}
If $\mathcal{X}$ is compact, Algorithm \ref{alg:ogo} will converge to an $\epsilon-$local optimal solution as $N \rightarrow \infty$.
\end{thm}
\begin{proof}
By the design of the algorithm, it only checks the optimality condition with hyperplanes within the corresponding convex region respect to the best known solution. Since $\mathcal{L}(\cdot, \hat{f})$ is convex and $\mathcal{X}$ is compact, the algorithm converges to an $\epsilon-$local optimal\cite{Kelley60}.  
\end{proof}
Note that the feasible set $\mathcal{X}$ does not need to be convex in order to prove the local optimality \cite{Geoffrion72,Eaves71}. The current representation of the algorithm retains all prior hyperplanes. The algorithm can be improved by deleting prior hyperplanes \cite{Hogan73}. 

Checking all possible activation patterns in $\mathcal{M}(\x^*)$ guarantees an $\epsilon-$local optimal. However, Set $\mathcal{M}(\x^*)$ can have exponentially many possible patterns, which leads to computational challenges. Alternatively, the algorithm can examine the previously visited patterns and add a termination check at Line \ref{alg:exit} by solving the outer-approximation (NOA). If the bound from the approximation is close to the best known objective, then the algorithm terminates. Otherwise, the algorithm continues with the solution of approximation (NOA). While this approach would reduce the computational challenges of examining on all possible patterns, the approach does not guarantee a local optimality in some cases. For example, the boundary solution $x=4$ in Figure \ref{fig:lopt} is the case if all previous observations are strictly greater than 4. 

\subsection{Gradient Projection Methods}
The gradient projection methods can solve a class of constrained neural network inverse problems. The gradient projection methods are an iterative method that identifies a descent direction with gradients and ensures the feasibility through projection operations \cite{Bertsekas99}. An iteration of gradient projection methods can be described as follows;
\begin{eqnarray}
x^{k+1} &=& x^k + \alpha^k ( \bar{x}^k - x^k ), \label{eq:pg_step}\\
\bar{x}^k &=& [ x^k - s^k \nabla g(x^k) ]^+_\mathcal{X}, \label{eq:pg_target}
\end{eqnarray}
where $x^k$ is the solution at iteration $k$, $\nabla g(x^k)$ is the gradient of $x^k$, $[\cdot]^+_\mathcal{X}$ denotes projection on the set $\mathcal{X}$, and parameters $\alpha^k$ and $s^k$ are a step size and a positive scalar, respectively. Variable $\bar{x}^k$ is the projected solution of $x^k - s^k \nabla g(x^k)$ onto the set $\mathcal{X}$. For a given solution $x^k$, Equation \eqref{eq:pg_target} finds a descent direction $x^k - s^k \nabla g(x^k)$ and if the resulting solution is outside of feasible region $\mathcal{X}$, it projects the solution onto the feasible region $\mathcal{X}$. Similar to gradient based algorithms, Equation \eqref{eq:pg_step} uses a step size to update the solution. 

The projection operation can be computationally expensive depending on the characteristics of feasible region $\mathcal{X}$. If the feasible region is simple bound constraints, the projection operation is same as rounding operation. For other cases, the projection operation can be described as a quadratic programming problem, which  often requires computational efforts. 

\section{Computational experiments} \label{sec:comp}
This section discusses the computational results against two sets of neural network inverse problem instances; One set is based on a material design problem. The other instances are randomly generated. 

All algorithms are implemented in Python 3.6.5 along with Tensorflow 1.10.0. The linear programming problems are implemented in PuLP 1.6.8\footnote{LP modeler written in Python {https://github.com/coin-or/pulp}}. The main linear programming solver of the implementation is CbC 2.9.0\footnote{https://projects.coin-or.org/Cbc}, which is the default linear solver for PuLP. Computational experiments have been conducted on a Sandy Bridge dual-socket Linux machine with eight 2.6 GHz cores on each socket and 64GB of RAM.

\subsection{Material design instances} \label{sec:mdi}
The material design problem considered in this section is to find the optimal topological polymer structure for given desired rheological properties. The main approach consists of two phases; the first phase is developing a neural network that mimics the behavior of the forward simulation, bob-rheology\footnote{https://sourceforge.net/projects/bob-rheology/}, that predicts the rheological properties with a given topological structure \cite{Read11}. The second phase is solving an inverse optimization problem for a given target rheological property description with the trained neural network. The main goal of the second phase is to find many good configurations rather than find a good configuration. Therefore, the multi-start approach is adopted.   

The best performing neural network has 5 dense layers with the structure of $7 \times 64 \times 128 \times 256 \times 534$ neurons. The first 7 neurons are inputs for the design parameters and the outputs of the last 534 neurons depicts the rheological properties. All intermediate neurons have a rectified linear activation unit for their output. The inverse optimization problem has two types of constraints. One is upper and lower bounds of input parameters. The other constraint is a ratio constraint of two input parameters such that the sum of the two input elements should be equal to 1.  

\begin{figure}[h]
\begin{center}
\includegraphics[width=0.4\textwidth]{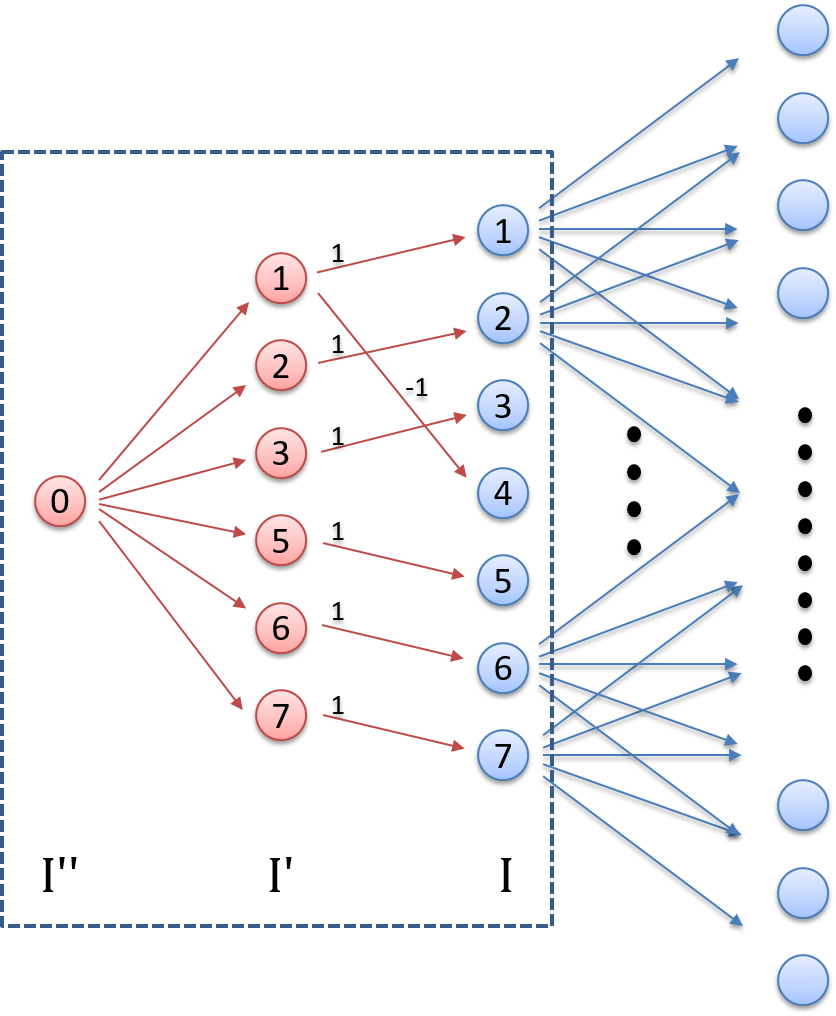}
\end{center}
\caption{A diagram of the modified neural network: The neurons and arcs in blue are from the trained network and the rest are introduced to solve the inverse problem within the Tensorflow framework. Note that the only trainable parameters are biases in layer $I'$. The rest of weights and biases are fixed.} \label{fig:mnn}
\end{figure}

The projected gradient method is implemented within Tensorflow with the following modification. The main modification is introducing two additional layers prior to the original input layer in order to describe the ratio constraint and describe the input parameters as variables. Let $I$ be the input layer of the original trained neural network. Let $I'$ be the additional layer connected the input layer $I$ and $I''$ be another layer adjacent to the layer $I'$. We create one fewer neurons in layer $I'$ such that we duplicate all neurons in layer $I$ except one neuron in the ratio constraint. In layer $I''$, we create only one neuron. We set the biases of layer $I$ to 0 except that the bias of the neuron without a replication in layer $I'$ is set to 1. The weights from layer $I'$ to layer $I$ are set to 1 if both neurons have a same neuron number. The weight from the ratio neuron in layer $I'$ to the ratio neuron in layer $I$ is set to -1. All other weights from layer $I'$ to layer $I$ are set to 0. The weights from neurons in layer $I''$ to neurons in layer $I'$ are set to 0. The neuron in layer $I''$ is the input layer of the modified neural network. All neurons in layer $I''$, $I'$ and $I$ are a linear system without any activation function. The rest of neural network remains same as the original trained neural network. At last, we set only the biases of neurons in layer $I'$ trainable and set all other weights and biases not trainable. Figure \ref{fig:mnn} depicts the modified neural network. Neurons and weight parameters in the dotted box are modified compared to the original trained neural network. 

The bound constraints can be incorporated with a simple projection procedure. First, we optimize the biases in layer $I'$ of the modified neural network with a Tensorflow optimizer. With a fixed number of iterations, we examine the solution if it is within the bounds. If it is outside of the bounds, we project it to the feasible bounds. Since it is a simple bound, the projection is equivalent to move the value to the closest bound. 

The projected gradient algorithm has two termination criteria; one condition is when the next solution is same as the previous solution, which means either the gradient of the current solution is equal to zero or the descent direction is infeasible respect to the bounds. The other is stalling behaviors such that the Tensorflow cannot find a better solution with multiple attempts. 

\subsubsection{Experimental design}
The projected gradient approach and the proposed outer-approximation guided approach have been tested on 100 instances. All instances share the same trained neural network discussed in Section \ref{sec:mdi} with a given target and the mean squared error as the loss function. The only difference between instances is their starting points, which have been collected through examination of the training data set. For each training set, it can be easily to measure the loss between the corresponding output and the given target. The following results are a summary of the 100 optimization runs. 

Table \ref{tab:optpara} summarizes the various optimization hyper-parameter configurations for the experiment. For the projected gradient method, Adam and RMSprop have been tested as the Tensorflow optimizer. For these specific instances, RMSprop outperforms the other optimizer. For the stalling termination criteria, the number of no improvement steps is set to 20, which is tuned to a balanced setting between the solution quality and the computational time. In order to improve the solution time of the projected gradient method, the projection operation is conducted at every 16 (epochs) steps of Tensorflow optimization instead of projection at each step. Table \ref{tab:optpara} (b) summarizes various step sizes that have been tested for the projected gradient method. Table \ref{tab:optpara} (a) summarizes various step size parameter $\gamma^s$ and neighborhood size parameter $\gamma^c$ of problem (DLOA) for the proposed outer-approximation guided method. All other hyper-parameters for Algorithm \ref{alg:ogo} are summarized in Table \ref{tab:optcfg}.
   
\begin{table}
\parbox[t]{.45\linewidth}{
\centering
\begin{tabular}[t]{l|c|c}
\multicolumn{1}{c|}{Label} & Step size ($\gamma^s$) & Neig. size ($\gamma^c$) \\
\hline \hline
OGO-0.01-0.1 & $10^{-2}$ & $10^{-1}$ \\
OGO-0.01-0.05 & $10^{-2}$ & $5 \times 10^{-2}$ \\
OGO-0.01-0.01 & $10^{-2}$ & $10^{-2}$ \\
OGO-0.1-0.5 & $10^{-1}$ & $5 \times 10^{-1}$ \\
OGO-0.1-0.05 & $10^{-1}$ & $5 \times 10^{-2}$ \\
\hline 
\end{tabular}
} \hfill
\parbox[t]{.45\linewidth}{
\centering
\begin{tabular}[t]{l|c}
\multicolumn{1}{c|}{Label} & Step size \\
\hline \hline
PGO-0.01  & $10^{-2}$  \\
PGO-0.001 & $10^{-3}$  \\
PGO-0.0001 & $10^{-4}$ \\
\hline 
\end{tabular}
}

\vspace{0.3cm}
\parbox{.45\linewidth}{
\centering
{(a) Outer approximation guided method}
} \hfill
\parbox{.45\linewidth}{
\centering
{(b) Projected gradient method}
}
\caption{Parameter settings for computational experiments} \label{tab:optpara}
\end{table}

\begin{table}[h] 
\begin{center}
\begin{tabular}{c|c|c|c|c|c}
$\epsilon$ & $\gamma^s_{\min}$ & $\gamma^s_{\max}$ & $\rho^c$ & $\rho^e$ & $N$\\
\hline
$1E-5$ & $\gamma^c$ & $\epsilon \times \sqrt{n}$ & $0.9$ & $1.5$ & $1E3$
\end{tabular} 
\caption{Optimization hyper-parameters for Algorithm \ref{alg:ogo}. Parameter $n$ denotes the number of input variables.} \label{tab:optcfg}
\end{center}
\end{table}
The percent-gap-closed to the best known solution are employed as the solution quality metrics. The percent-gap-closed metrics $\rho^k_j$ for instance $j$ and approach $k$ is defined as follows; 
\begin{equation}
\rho^k_j = \frac{v^0_j - v^k_j}{v^0_j - v^*_j}, \label{eq:pgc}
\end{equation} 
where $v^0_j$ and $v^*_j$ denote the initial objective value and the best objective value among all approaches for instance $j$. $v^k_j$ is the objective value of approach $k$ for instance $j$. The denominator of Equation \eqref{eq:pgc} computes the best known improvement for instance $j$ and the numerator computes the improvement for the corresponding approach. The metrics measures how close to the best known solution the solution from each method is.

\subsubsection{Overall solution time and quality}
\begin{figure}[h]
\begin{center}
\includegraphics[width=0.6\textwidth]{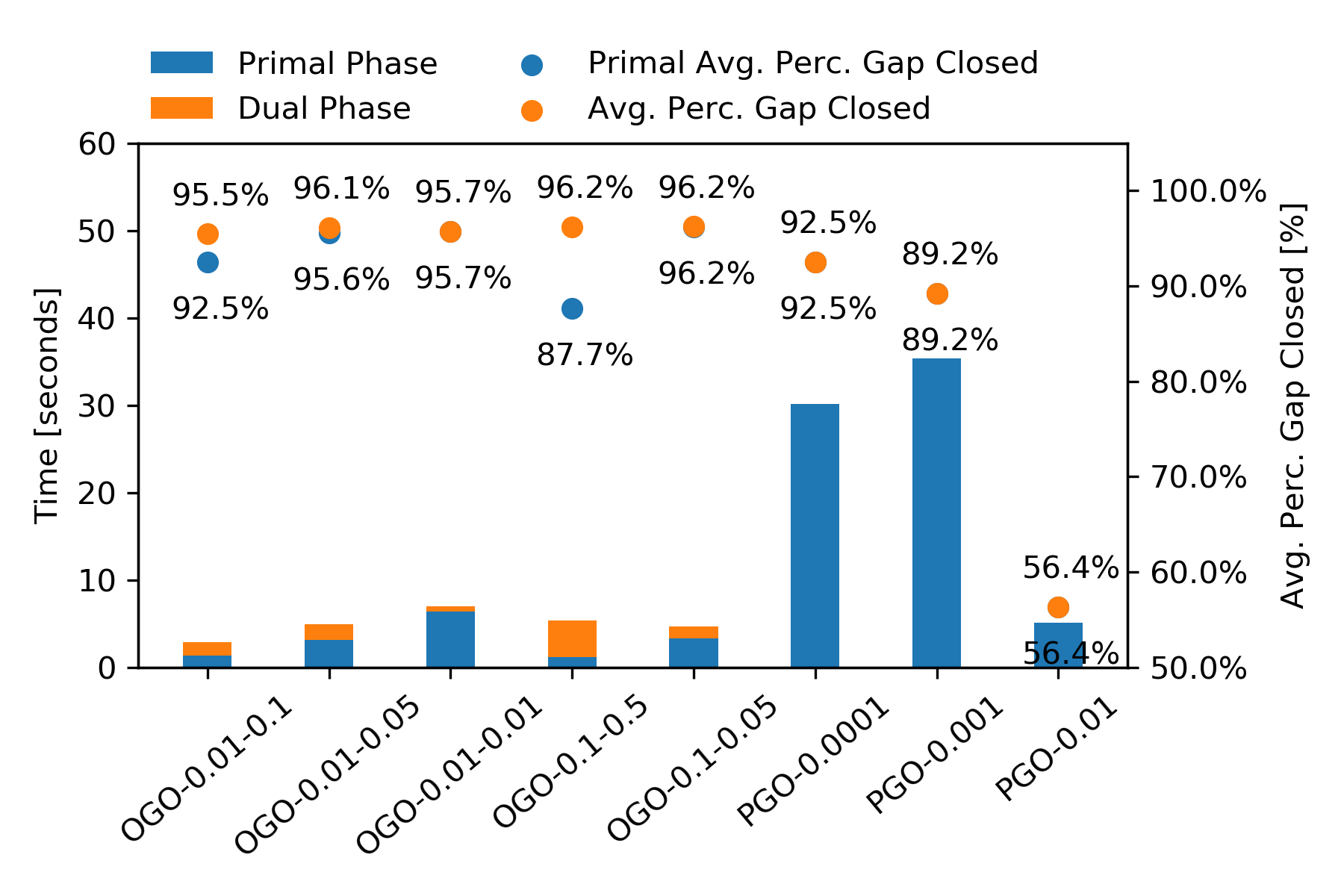}
\end{center}
\caption{The average solution time and average percent-gap-closed metrics. The x-axis represents optimization approaches with various hyper-parameter settings. The bar charts are the average solution time of 100 instances. The dot charts show the average percent-gap-closed across 100 optimization runs at termination. The blue bars and dots are the results for the primal phase and the orange bars and dots are corresponding to the dual phase.} \label{fig:avg_t_p}
\end{figure}

Figure \ref{fig:avg_t_p} summarizes the average solution time and the average percent-gap-closed metrics across 100 instances in the primal and dual phases. The proposed outer-approximation algorithm can alter the status between primal and dual phases. However, in the computational summary, once the algorithm enters the dual phase, the subsequent procedure is treated as the dual phase. The projected gradient method has only the primal phase. 

For the projected gradient method, as the step size decreases, the solution time increases and the solution quality is improved. There is no such clear trend for the proposed algorithm in terms of the solution quality and time. The computational result implies that the time spent in the primal phase is dependent on the ratio of the step size $\gamma^s$ and the neighborhood size $\gamma^c$. If they are similar, the algorithm tends to spend more time in the primal phase. 

The proposed outer-approximation guided method outperforms the projected gradient method in terms of the solution time and quality. The projected gradient method with a small step size (PGO-0.01) can solve the problems within  comparable solution times as the proposed algorithm but it cannot achieve similar solution qualities. The solution quality from the proposed approach is superior to one from the projected gradient method. 

The average improvement on the computational time ranges between 30 and 35 seconds, which can be considered as insignificant. However, in the material design context, it is required to solve many cases up to hundreds or thousands in order to produce many leading candidates and overcome the local optimality issues. In that regard, the percentage improvement, 1 - the average time for the proposed algorithm / the average time for the projected gradient method, is a more meaningful measure, which ranges between 77\% and 92\%. It implies that the proposed algorithm can solve the same number of problems with less than 23\% of the computational efforts.    

\subsubsection{Solution progress profile}
\begin{figure}[h]
\begin{center}
\includegraphics[width=0.6\textwidth]{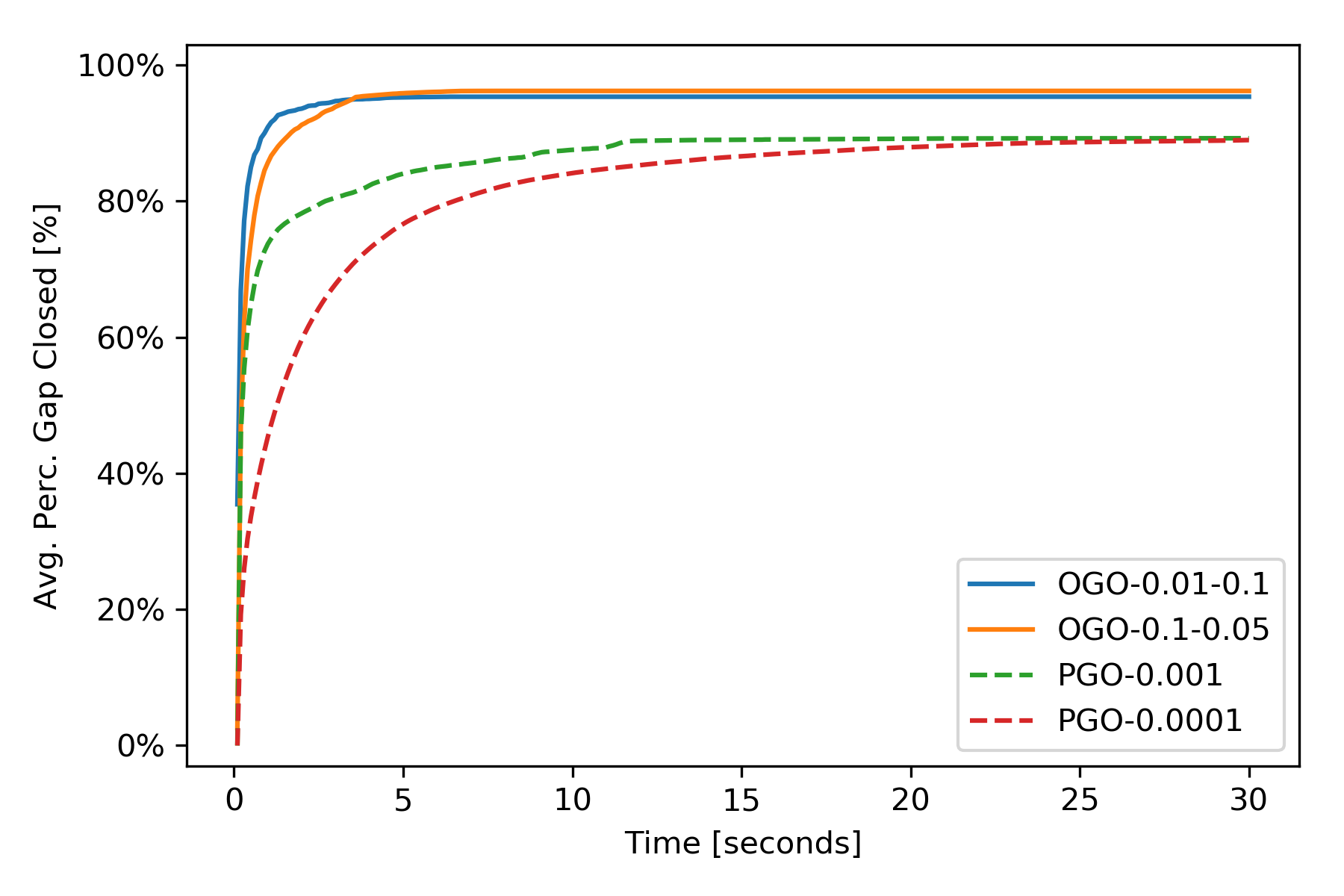}
\end{center}
\caption{The progress profile of the average percentage gap closed metrics. The x-axis represents time spent in seconds. The y-axis is the average percentage gap closed.} \label{fig:ogo_vs_pgo}
\end{figure}

Figure \ref{fig:ogo_vs_pgo} shows the progress profile of the solution for four approaches (OGO-0.01-0.1, OGO-0.1-0.05, PGO-0.001, PGO-0.0001). The graphs in Figure \ref{fig:ogo_vs_pgo} show the average solution quality over time. It clearly shows that the proposed algorithm finds a better solution quickly. Both outer-approximation approaches achieve over 90\% solution quality within 5 seconds whereas the projected gradient methods around 80\% of the best known solution.

\subsection{Randomly generated instances}
The random instances are generated as follows; The process starts with a given neural network structure. The first step is randomly generating weights and biases for the given neural network structure. The initial randomization uses a normal distribution. Randomly generated neural networks tend to have a huge spread at outputs even with controlled inputs ranging between 0 and 1. The second step is to scale the output of the neural network. In the second step, the neural network is re-trained with many scaled outputs, which can be generated from the neural network from the first step. Many pairs of input and output data can be generated along with the neural network from the first step. After the output data is processed to range between 0 and 1, the processed data is used to retrain the neural network. 

The two network architectures in Table \ref{tab:brnn} are considered as the base structures. The targets are randomly generated and are not part of the scaled training set. The starting points are randomly chosen with a standard uniform distribution. Only bound constraints for inputs are considered.
\begin{table}[h]
\begin{center}
\begin{tabular}{r|l}
\hline 
S1 & 100,256,128,64,128,64,32,8 \\
\hline
S2 & 256,128,64,128,64 \\
\hline
\end{tabular}
\caption{Base neural network structures for random instances} \label{tab:brnn}
\end{center}
\end{table} 

For the outer approximation guided approach, the hyper-parameter settings in Table \ref{tab:rndoptcfg} are tested. In this experiment, the various optimality gap criteria are considered. For the projected gradient approach, two step sizes are considered. PGO\_0.001 and PGO\_0.0001 refer to the projected gradient methods with step size $0.001$ and $0.0001$, respectively.  

\begin{table}[h] 
\begin{center}
\begin{tabular}{l|c|c|c|c|c|c|c|c}
\multicolumn{1}{c|}{Label} & $\epsilon$ & $\gamma^s$ & $\gamma^c$ & $\gamma^s_{\min}$ & $\gamma^s_{\max}$ & $\rho^c$ & $\rho^e$ & $N$\\
\hline
OGO\_0.001 & $1E-3$ & $0.01$ & $0.5$ & $\gamma^c$ & $\epsilon \times \sqrt{n}$ & $0.9$ & $1.5$ & $1E3$ \\
\hline
OGO\_0.0001 & $1E-4$ & $0.01$ & $0.5$ & $\gamma^c$ & $\epsilon \times \sqrt{n}$ & $0.9$ & $1.5$ & $1E3$
\end{tabular} 
\caption{Optimization hyper-parameters of Algorithm \ref{alg:ogo} for random instances. Parameter $n$ denotes the number of input variables.} \label{tab:rndoptcfg}
\end{center}
\end{table}

Because of random starting points, the percentage-gap-closed metrics does not distinguish the performance difference of approaches. The most of final solutions are within a range of $10^{-2}$ and initial solutions are typically far from the local optimal solution because of random starting points. These two factors make the denominator of the metrics large and the differences in the numerator small. In this experiment, the absolute difference to the best known solution is used instead of the percentage-gap-closed. The metrics measures the absolute difference between the best known solution and the solution of each instance and approach.        

\begin{figure}[h]
\begin{center}
\begin{tabular}{cc}
\includegraphics[width=0.4\textwidth]{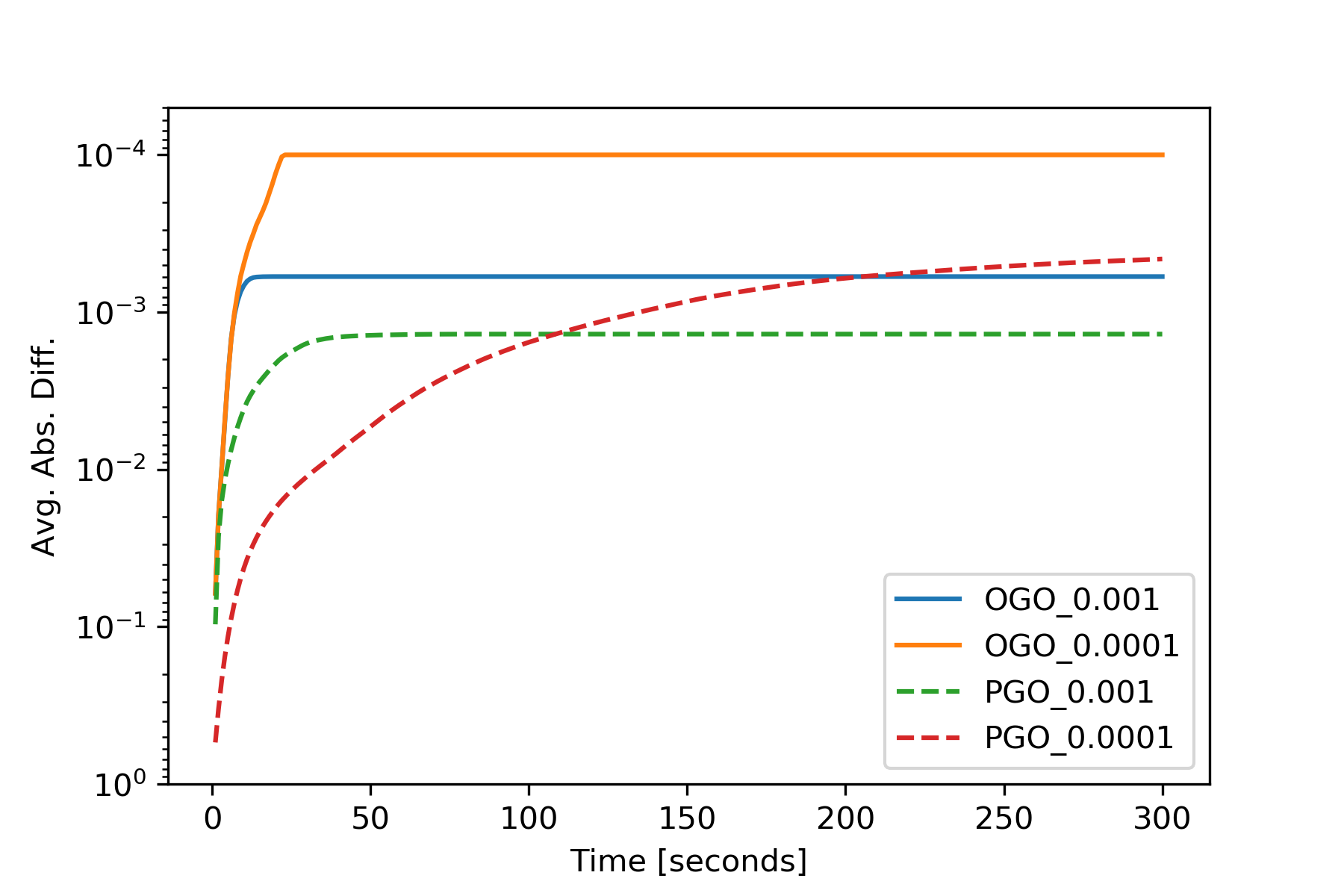} & \includegraphics[width=0.4\textwidth]{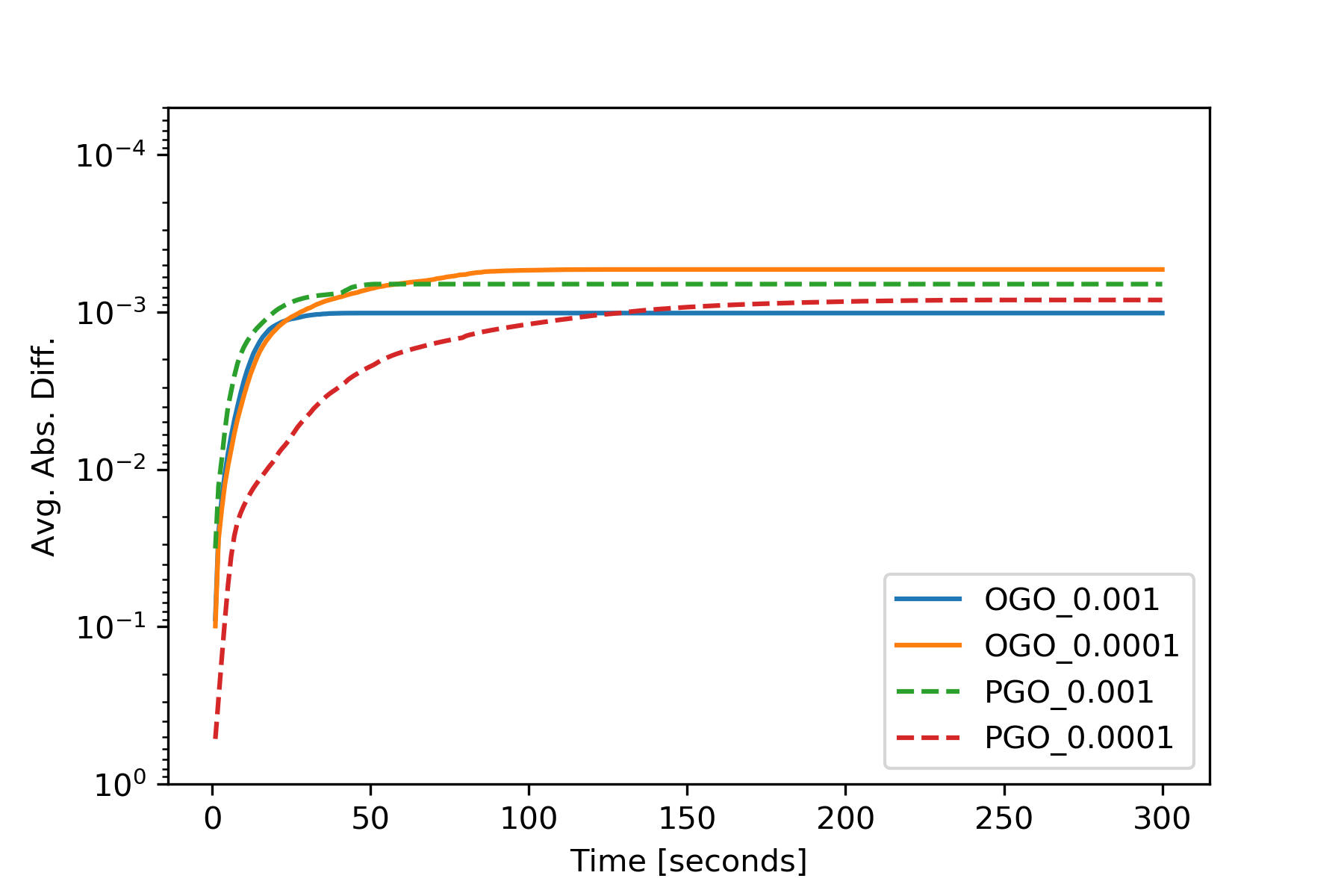} \\
(a) S1 & (b) S2
\end{tabular}
\end{center}
\caption{The progress profile of the average absolute difference. The x-axis represents time spent in seconds. The y-axis in a logarithmic scale is the average absolute difference to the best known solution.} \label{fig:rnd_prog}
\end{figure}

Figure \ref{fig:rnd_prog} shows the representative progress profiles of experiments. The y-axis in a logarithmic scale is the average absolute difference between the solution over time and the best known solution. Each line denotes the different approach. Note that the orange line is truncated at $10^{-4}$ for display purpose.

The outer approximation guided approach finds a good quality solution quickly for neural network structure S1 instances compared to the projected gradient method. For neural network structure S2 instances, the projected gradient method with step size of $0.001$ shows better performance than two outer-approximation guided approaches at the earlier stage of the progress. After around 50 seconds, the outer-approximation guided approach with optimality tolerance of $10^{-4}$ achieves a better average solution quality. Note that the optimality gap $\epsilon$ is respect to a local optimal solution and not for the global optimality. Therefore, OGO\_0.0001 can have solutions with more than $0.0001$ absolute difference. It is hard to say that the proposed algorithm always outperforms the projected gradient methods. The numerical results show that the proposed algorithm is generally superior to the projected gradient method discussed. 

\section{Conclusion}
This paper discusses constrained neural network inverse optimization problems that find the optimal set of input parameters for a desired output and proposes an outer-approximation guided method, especially when the inverse problem has additional constraints on input parameters. The proposed method is devised by exploiting the characteristics of the local optimal solution for neural network inverse optimization problems with rectified activation units. The proposed algorithm consists of primal and dual phases. In the primal phase, the algorithm incorporates neighbor gradients through outer approximations of neighbors to expedite the convergence to a good quality solution. In the dual phase, the algorithm exploits the convex structure of the local optimal solution to improve the speed of convergence. In addition, the paper proposes a method to generate feasibility cuts without solving an explicit dual problem. The superiority of the proposed algorithm is demonstrated with computational results compared to a projected gradient method for a material design problem and randomly generated instances.

\bibliographystyle{apalike}
\bibliography{dnngradopt}

\end{document}